\documentclass[final, a4paper]{amsart}
\usepackage{geometry}
\usepackage{amssymb}
\usepackage[english]{babel}
\usepackage{pifont, array}
\usepackage{csquotes}
 \usepackage{hyperref}
 \hypersetup{
    colorlinks=true,
    linkcolor=blue,
    citecolor=blue,
    linktoc=page,
    pdfauthor={Yohann de Castro}
    pdftitle={A remark on the lasso and the Dantzig selector}
  }
\newcommand{\Id}{\mathrm{Id}}
\renewcommand{\S}{\mathcal{S}}
\newcommand{\R}{\mathbb{R}}                     
\newcommand{\betastar}{\beta^\star}
\newcommand{\abs}[1]{\left\vert#1\right\vert}
\newcommand{\norm}[2]{\big\lVert#2\big\lVert_{#1}}
\newcommand{\supp}[1]{\mathrm{Supp}(#1)}
\newcommand{\RIP}{\mathrm{RIP}}
\newenvironment{pfof}[1]{\noindent{\bf Proof of
#1}\ {\bf ---}}
	{\hfill\qed\vspace{1ex}}
\newtheoremstyle{theo}
{}
{}
{\itshape}
{\parindent}
{\bf}
{\ ---}
{.5em}
{}%
\theoremstyle{theo}
\newtheorem{theorem}{Theorem}[section]
\newtheorem{lemma}[theorem]{Lemma}
\newtheorem{proposition}[theorem]{Proposition}

\newtheoremstyle{def}%
{}
{}
{\itshape}
{\parindent}
{\bf}
{\ ---}
{.5em}
{}%
\theoremstyle{def}
\newtheorem{definition}{Definition}[section]

\theoremstyle{remark}
\newtheorem*{remark}{Remark}

\begin{document}
\pagestyle{headings}
\title{A remark on the lasso and the Dantzig selector}
\date{\today}
 \keywords{Lasso; Dantzig selector; Oracle inequality; Almost-Euclidean section; Distortion.}
\author{Yohann de Castro}
\address{Yohann de Castro is with the D\'epartement de Math\'ematiques, Universit\'e Paris-Sud, Facult\'e des Sciences d'Orsay, 91405 Orsay, France.
This article was written mostly during his Ph.D. at the Institut de Math\'ematiques de Toulouse.}
\email{yohann.decastro@math.u-psud.fr}
\begin{abstract}
This article investigates a new parameter for the high-dimensional regression with noise: \textit{the distortion}. This latter has attracted a lot of attention recently with the appearance of new deterministic constructions of ``almost''-Euclidean sections of the L1-ball. It measures how far is the intersection between the kernel of the design matrix and the unit L1-ball from an L2-ball. We show that the distortion holds enough information to derive \textit{oracle inequalities} (i.e. a comparison to an ideal situation where one knows the s largest coefficients of the target) for the lasso and the Dantzig selector. 

%

\end{abstract}
\maketitle
 \section{Introduction}
In the past decade much emphasis has been put on recovering a large number of unknown variables from few noisy observations. Consider the high-dimensional linear model where one observes a vector $y\in\R^n$ such that
\[
y=X\betastar+\varepsilon\,,
\]
where $X\in\R^{n\times p}$ is called the design matrix (known from the experimenter), $\betastar\in\R^p$ is an unknown target vector one would like to recover, and $\varepsilon\in\R^n$ is a stochastic error term that contains all the perturbations of the experiment. 

A standard hypothesis in high-dimensional regression \cite{hastie2009high} requires that one can provide a constant $\lambda^0_n\in\R$, as small as possible, such that
\begin{equation}\label{Borne Az}
\norm{\ell_\infty}{X^\top \varepsilon}\leq \lambda^0_n,
\end{equation}
with an overwhelming probability, where $X^\top\in\R^{p\times n}$ denotes the transpose matrix of $X$. In the case of $n$-multivariate Gaussian distribution, it is known that $\lambda^0_n = \mathcal O(\sigma_n\sqrt{{\log p}})$, where $\sigma_n>0$ denotes the standard deviation of the noise; see Lemma \ref{Bound on the noise}.

Suppose that you have far less observation variables $y_i$ than the unknown variables $\betastar_i$. For instance, let us mention the \emph{Compressed Sensing} problem \cite{MR2241189,MR2230846} where one would like to simultaneously acquire and compress a signal using few (non-adaptive) linear measurements, i.e. $n\ll p$. In general terms, we are interested in accurately estimating the target vector $\betastar$ and the response $X\betastar$ from few and corrupted observations. During the past decade, this challenging issue has attracted a lot of attention among the statistical society. In 1996, R. Tibshirani introduced the lasso \cite{MR1379242}:
\begin{equation}\label{Lasso}
 \beta^\ell\in\displaystyle\mathrm{arg}\min\limits_{\beta\in\R^p}\Big\{\frac12\norm{\ell_2}{
X\beta-y }^2+\lambda_\ell \norm{\ell_1}{\beta}\Big\},
\end{equation}
where $\lambda_\ell>0$ denotes a tuning parameter. Two decades later, this estimator continues to play a key role in our understanding of high-dimensional inverse problems. Its popularity might be due to the fact that this estimator is computationally tractable. Indeed, the lasso can be recasted in a Second Order Cone Program (SOCP) that can be solved using an interior point method. Recently, E.J. Cand\`es and T. Tao \cite{MR2382644} have introduced the Dantzig selector as
\begin{equation}\label{Dantzig}
 {\beta}^d\in\arg \min_{\beta\in\R^p}
\norm{\ell_1}{\beta}\ \ \mathrm{s.t.}\ \
\lVert{X^\top(y-X\beta)}\lVert_{\ell_\infty}\leq\lambda_d\,,
\end{equation}
where $\lambda_d>0$ is a tuning parameter. It is known that it can be recasted as a linear program. Hence, it is also computationally tractable. A great statistical challenge is then to find efficiently verifiable conditions on $X$ ensuring that the lasso \eqref{Lasso} or the Dantzig selector \eqref{Dantzig} would recover \enquote{most of the information} about the target vector $\betastar$.

\subsection{Our goal}
What do we precisely mean by \enquote{most of the information} about the target? What is the amount of information one could recover from few observations? These are two of the important questions raised by Compressed Sensing. Suppose that you want to find an $s$-sparse vector (i.e. a vector with at most $s$ non-zero coefficients) that represents the target, then you would probably want that it contains the $s$ largest (in magnitude) coefficients $\betastar_i$. More precisely, denote by $\mathcal S_\star\subseteq \{1,\dotsc,p\}$ the set of the indices of the $s$ largest coefficients. The $s$-best term approximation vector is $\betastar_{\mathcal S_\star}\in\R^p$ where $(\betastar_{\mathcal S_\star})_i=\betastar_i$ if $i\in\mathcal S_\star$ and $0$ otherwise. Observe that it is the $s$-sparse projection in respect to any $\ell_q$-norm for $1\leq q<+\infty$ (i.e. it minimizes the $\ell_q$-distance to $\betastar$ among all the $s$-sparse vectors), and then the most natural approximation by an $s$-sparse vector.

Suppose that someone gives you all the keys to recover $\betastar_{\S_\star}$. More precisely, imagine that you know the subset $\mathcal S_\star$ in advance and that you observe $y^{oracle} = X\betastar_{\mathcal S_\star}+\varepsilon$. This is an ideal situation referred as the oracle case. Assume that the noise $\varepsilon$ is a Gaussian white noise of standard deviation $\sigma_n$, i.e. $\varepsilon\sim\mathcal N_n(0,\sigma_n^2\,\Id_n)$ where $\mathcal N_n$ denotes the $n$-multivariate Gaussian distribution. Then the optimal estimator is the ordinary least square $\beta^{idea\ell}\in\R^p$  on the subset $\mathcal S_\star$, namely:
\begin{equation*}
  \beta^{idea\ell}\in\mathrm{arg}\min\limits_{\substack{\beta\in\R^p\\
\supp{\beta}\subseteq\mathcal S_\star}}\norm{\ell_2}{X\beta-y^{oracle}}^2,
\end{equation*}
where $\supp\beta\subseteq\{1,\dotsc,p\}$ denotes the support (i.e. the set of the indices of the non-zero coefficients) of the vector $\beta$. It holds
\[
\norm{\ell_1}{\beta^{idea\ell}-\betastar}=\norm{\ell_1}{\beta^{idea\ell}-\betastar_{\mathcal
S_\star}}+\norm{\ell_1}{\betastar_{\mathcal S_\star^c}}\leq\sqrt
s\norm{\ell_2}{\beta^{idea\ell}-\betastar_{\mathcal
S_\star}}+\norm{\ell_1}{\betastar_{\mathcal S_\star^c}}\,,
\]
where $\betastar_{\mathcal S_\star^c}=\betastar-\betastar_{\mathcal S_\star}$ denotes the error vector of the $s$-best term approximation. A calculation of the solution of the least square estimator shows that:
\begin{align*}
 \mathbb E\norm{\ell_2}{\beta^{idea\ell}-\betastar_{\mathcal S_\star}}^2 & = \mathbb E\norm{\ell_2}{\big(X_{\S_\star}^\top X_{\S_\star}\big)^{-1}X_{\S_\star}^\top y^{oracle}-\betastar_{\mathcal
S_\star}}^2\,,\\
& = \mathbb E\norm{\ell_2}{\big(X_{\S_\star}^\top X_{\S_\star}\big)^{-1}X_{\S_\star}^\top \varepsilon}^2 = \mathrm{Trace}\big(\big(X_{\S_\star}^\top
X_{\S_\star}\big)^{-1}\big)\cdot {\sigma_n^2}\,,\\
&\geq\Big(\frac
1{\rho_1}\Big)^2\cdot{\sigma_n^2}\cdot s\,,
\end{align*}
where $X_{\S_\star}\in\R^{n\times s}$ denotes the matrix composed by the columns $X_i\in\R^n$ of the matrix $X$ such that $i\in\S_\star$, and $\rho_1$ is the largest singular value of $X$. It yields that
\[
\Big[\mathbb E \norm{\ell_2}{\beta^{idea\ell}-\betastar_{\mathcal S_\star}}^2\Big]^{1/2}\geq\frac1{\rho_1}\cdot {\sigma_n}\cdot \sqrt s.
\]
In a nutshell, the $\ell_1$-distance between the target $\betastar$ and the optimal estimator $\beta^{idea\ell}$ can be reasonably said of the order of
\begin{equation}\label{Oracle variable selection}
\frac1{\rho_1}\cdot {\sigma_n}\cdot  s +
\norm{\ell_1}{\betastar_{\mathcal S_\star^c}}.
\end{equation}
In this article, we say that the lasso satisfies a \textit{variable selection oracle inequality of order $s$} if and only if its $\ell_1$-distance to the target, namely $\norm{\ell_1}{\beta^\ell-\betastar}$, is bounded by \eqref{Oracle variable selection} up to a \enquote{satisfactory} multiplicative factor.

In some situations it could be interesting to have a good approximation of $X\betastar$. In the oracle case, we have
\begin{align*}
 \norm{\ell_2}{X\beta^{idea\ell}-X\betastar}&\leq
\norm{\ell_2}{X\beta^{idea\ell}-X\betastar_{\S_\star}}+\norm{\ell_2}{X\betastar_{\S_\star^c}}\,,\\
 &\leq \norm{\ell_2}{X\beta^{idea\ell}-X\betastar_{\S_\star}} + \rho_1
\norm{\ell_1}{\betastar_{\S_\star^c}}\,.
\end{align*}
where $\rho_1$ denotes the largest singular value of $X$. An easy calculation gives that
\[\mathbb E\norm{\ell_2}{X\beta^{idea\ell}-X\betastar_{\mathcal
S_\star}}^2
=\mathrm{Trace}\big(X_{\S_\star}\big(X_{\S_\star}^\top
X_{\S_\star}\big)^{-1}X_{\S_\star}^\top\big)\cdot {\sigma_n^2}={\sigma_n^2}\cdot s.\]
Hence a tolerable upper bound is given by
\begin{equation}\label{Error prediction selection}
 {\sigma_n}\cdot \sqrt s + \rho_1
\norm{\ell_1}{\betastar_{\S_\star^c}}.
\end{equation}
We say that the lasso satisfies an \textit{error prediction oracle inequality of order $s$} if and only if its prediction error is upper bounded by \eqref{Error prediction selection} up to a \enquote{satisfactory} multiplicative factor (say logarithmic in $p$).

\subsection{Framework}

In this article, we investigate designs with known distortion. We begin with the definition of this latter:
\begin{definition}
 A subspace $\Gamma\subset\R^p$ has a distortion $1\leq\delta\leq\sqrt p$ if and only if
\[\forall x\in\Gamma,\quad\norm{\ell_1}{x}\leq\sqrt p\,\norm{\ell_2}{x}\leq \delta
\norm{\ell_1}{x}.\]
\end{definition}

\noindent A long standing issue in approximation theory in Banach spaces is to find \enquote{almost}-Euclidean sections of the unit $\ell_1$-ball, i.e. subspaces with a distortion
$\delta$ close to $1$ and a dimension close to $p$. In particular, we recall that it has been established \cite{MR0481792} that, with an overwhelming probability, a random subspace 
of dimension $p-n$ (with respect to the Haar measure on the Grassmannian) satisfies
\begin{equation}\label{Best Distortion}
 \delta \leq C\,\bigg(\frac {p(1+\log(p/n))}n\bigg)^{1/2}
\end{equation}
where $C>0$ is a universal constant. In other words, it was shown that, for all $n\leq p$, there exists a subspace $\Gamma_n$ of dimension $p-n$ such that, for all $x\in\Gamma_n$,
\[
\norm{\ell_2}{x}\leq C\,\bigg(\frac {1+\log(p/n)}n\bigg)^{1/2}\,
\norm{\ell_1}{x}.
\]
\begin{remark}
 Hence, our framework deals also with unitary invariant random matrices. For instance, the matrices with i.i.d. Gaussian entries. Observe that their distortion satisfies \eqref{Best Distortion}.
\end{remark}

Recently, new \textbf{deterministic} constructions of ``almost''-Euclidean sections of the $\ell_1$-ball have been given.  Most of them can be viewed as related to the context of error-correcting codes. Indeed, the construction of \cite{indyk2007uncertainty} is based on amplifying the minimum distance of a code using expanders. While the construction of
\cite{guruswami2008almost} is based on Low-Density Parity Check (LDPC) codes. Finally, the construction of \cite{indyk2010almost} is related to the tensor product of error-correcting codes. The main reason of this surprising fact is that the vectors of a subspace of low distortion must be \enquote{well-spread}, i.e. a small subset of its coordinates cannot contain most of its $\ell_2$-norm (cf \cite{indyk2007uncertainty,guruswami2008almost}). This property is required from a good error-correcting code, where the weight (i.e. the $\ell_0$-norm) of each codeword cannot be concentrated on a small subset of its coordinates. Similarly, this property was intensively studied in Compressed Sensing; see for instance the Nullspace Property in \cite{MR2449058}.
\begin{remark}
 The main point of this article is that all of these deterministic constructions give efficient designs for the lasso and the Dantzig selector.
\end{remark}

\subsection{The Universal Distortion Property}

In the past decade, numerous conditions have been given to prove oracle inequalities for the lasso and the Dantzig selector. An overview of important conditions can be found in \cite{MR2576316}. We introduce a new condition, the Universal Distortion Property (UDP).
\begin{definition}[$\mathrm{UDP}(S_0,\kappa_0, \Delta )$]
Given $1\leq S_0\leq p$ and $0<\kappa_0<1/2$, we say that a matrix
$X\in\R^{n\times p}$ satisfies the universal distortion condition of order
$S_0$,
magnitude $\kappa_0$ and parameter $\Delta$ {\bf if and
only if} for all $ \gamma\in\R^p $, for all integers $ s\in\{1,\dotsc,S_0\}$, for all
subsets
$\S\subseteq\{1,\dotsc,p\}$ such that $\abs\S=s$, it holds
\begin{equation}\label{Def UDP}
\norm{\ell_1}{\gamma_\S}\leq\Delta\sqrt
s\,\norm{\ell_2}{X \gamma}+\kappa_0\norm{\ell_1}{\gamma}.
\end{equation}
\end{definition}
\noindent 
\begin{remark}
\noindent -- Observe that the design $X$ is not normalized. Equation \eqref{parameters UDP} in Theorem \ref{Lemma Universality} shows that $\Delta$ can depend on the inverse of the smallest singular value of $X$. Hence the quantity $\Delta\norm{\ell_2}{X \gamma}$ is scalar invariant.\\
 \noindent -- The UDP condition is similar to the Magic Condition \cite{MR2779396} and the Compatibility Condition \cite{MR2576316}.
\end{remark}
\noindent
The main point of this article is that UDP is verifiable \textit{as soon as one can give an upper bound on the distortion of the kernel of the design matrix}; see Theorem \ref{Lemma Universality}. Hence, instead of proving that a sufficient condition (such as RIP \cite{MR2230846}, REC \cite{MR2533469}, Compatibility \cite{MR2576316}, ...) holds it is sufficient to compute the distortion and the largest singular value of the design. Especially as these conditions can be hard to prove for a given matrix. We recall that an open problem is to find a computationally efficient algorithm that can tell if a given matrix satisfies the RIP condition \cite{MR2230846} or not.


We call the property \enquote{Universal Distortion} because it is satisfied by all the full rank
matrices (Universal) and the parameters $S_0$ and $\Delta$ can be expressed in terms of the
distortion of the kernel $\Gamma$ of $X$:
\begin{theorem}\label{Lemma interpolation}
Let $X\in\R^{n\times p}$ be a full rank matrix. Denote by $\delta$ the distortion of its
kernel:
\[
 \delta=\sup_{\gamma\in\mathrm{ker}(X)}\frac{\norm{\ell_1}{\gamma}}{\sqrt p \norm{\ell_2}\gamma}\,,
\]
\noindent
and $\rho_n$ its smallest singular value. Then, for all $\gamma\in\R^p$,
\[
\norm{\ell_2}{\gamma}\leq \frac{\delta}{\sqrt
p}\norm{\ell_1}\gamma+\frac{2\delta}{\rho_n}\norm{\ell_2}{X\gamma}\,.
\]
Equivalently, we have $\mathcal{B}:=\{\gamma\in\R^p\ |\ ({\delta}/{\sqrt
p})\norm{\ell_1}{\gamma}+({2\delta}/{\rho_n})\norm{\ell_2}{X\gamma}\leq 1\}\subset B_2^p$, where $B_2^p$ denotes the Euclidean unit ball in $\R^p$.
\end{theorem}
\noindent
 This result implies that every full rank matrix satisfies UDP with
parameters described as follows.

\begin{theorem}\label{Lemma Universality}
 Let $X\in\R^{n\times p}$ be a full rank matrix. Denote by $\delta$ the distortion of its
kernel and $\rho_n$ its smallest singular value. Let $0<\kappa_0<1/2$ then $X$ satisfies
$\mathrm{UDP}(S_0,\kappa_0, \Delta)$ where
\begin{equation}\label{parameters UDP}
 S_0= \Big(\frac{\kappa_0}{\delta}\Big)^2p\quad\mathrm{and}
\quad\Delta=\frac{2\delta}{\rho_n}\,.
\end{equation}
\end{theorem}

\noindent This theorem is sharp in the following sense. The parameter $S_0$ represents
(see Theorem \ref{Theo12}) the maximum number of coefficients that can be recovered using
lasso, we call it the \textit{sparsity level}. It is
known \cite{MR2449058} that the best bound one could expect is $S_{opt}\approx n/\log(p/n)$,
up to a multiplicative constant. In the case where \eqref{Best Distortion} holds, the sparsity
level satisfies
\begin{equation}\label{eq:SparsityLevelDistortion}
 S_0\approx \kappa_0^2\ S_{opt}.
\end{equation}
It shows that any design matrix with low distortion satisfies UDP with an optimal
sparsity level.

\section{Oracle inequalities}
The results presented here fold into two parts. In the first part we assume only that UDP holds. In particular, it is not excluded that one can get better upper bounds on the parameters than Theorem \ref{Lemma Universality}. As a matter of fact, the smaller $\Delta$ is, the sharper the oracle inequalities are. Then, we give oracle inequalities in terms of only the distortion of the design.
\begin{theorem}\label{Theo12}
 Let $X\in\R^{n\times p}$ be a full column rank matrix. Assume that $X$ satisfies
$\mathrm{UDP}(S_0,\kappa_0, \Delta )$ and that \eqref{Borne Az} holds. Then for any
\begin{equation}\label{Borne parametre lambda}
 \lambda_\ell>{\lambda_n^0}/{(1-2\kappa_0)},
\end{equation}
it holds
\begin{equation}\label{ConsistEll1}
\norm{\ell_1}{\beta^\ell-\betastar}\leq\frac{2}{\big(1-\frac{\lambda_n^0}{\lambda_\ell}
\big)-2\kappa_0
} \
\min_{\substack{\S\subseteq\{1,\dotsc,p\},\\ \abs{\S}=s,\ s\leq S_0.}}
\Big(\lambda_\ell\,\Delta^2\,s+\norm{\ell_1}{\betastar_{\S^c}}\Big).
\end{equation}

\noindent
--- For every full column
rank matrix $X\in\R^{n\times p}$, for all $0<\kappa_0<1/2$ and $\lambda_\ell$ satisfying
\eqref{Borne
parametre lambda}, we have
\begin{equation}\label{ConsistEll12}
\norm{\ell_1}{\beta^\ell-\betastar}\leq\frac{2}{\big(1-\frac{\lambda_n^0}{\lambda_\ell}
\big)-2\kappa_0
} \
\min_{\substack{\S\subseteq\{1,\dotsc,p\},\\ \abs{\S}=s,\\ s\leq ({\kappa_0}/{\delta})^2p.}}
\Big(\lambda_\ell\cdot\frac{4\,\delta^2}{\rho_n^2}\cdot s+\norm{\ell_1}{
\betastar_ { \S^c } } \Big),
\end{equation}
where $\rho_n$ denotes the smallest singular value of $X$ and $\delta$ the distortion of its kernel.
\end{theorem}
\noindent \ding{71}
Consider the case where the noise satisfies the hypothesis of Lemma \ref{Bound on the noise} and
take $\lambda_n^0=\lambda_n^0(1)$. Assume that $\kappa_0$ is constant (say $\kappa_0=1/3$)
and take
$\lambda_\ell=3\lambda_n^0$; then \eqref{ConsistEll1} becomes
\[\norm{\ell_1}{\beta^\ell-\betastar}\leq12
\min_{\substack{\S\subseteq\{1,\dotsc,p\},\\ \abs{\S}=s,\ s\leq S_0.}}
\Big(6\,\norm{\ell_2,\infty}{X}\cdot\Delta^2\sqrt{\log
p}\cdot\sigma_n\, s+\norm{\ell_1}{ \betastar_{\S^c}}\Big),\]
which is an oracle inequality up to a multiplicative factor $\Delta^2\sqrt{\log
p}$. In the same way, \eqref{ConsistEll12} becomes
\[\norm{\ell_1}{\beta^\ell-\betastar}\leq12
\min_{\substack{\S\subseteq\{1,\dotsc,p\},\\ \abs{\S}=s,\\ s\leq p/9{\delta}^2.}}
\Big(24\,\norm{\ell_2,\infty}{X}\cdot\frac{\delta^2\sqrt{\log
p}}{\rho_n}\cdot\frac1{\rho_n}\,\sigma_n\, s+\norm{\ell_1}{ \betastar_{\S^c}}\Big),\]
which is an oracle inequality up to a multiplicative factor $C_{mult}:=({\delta^2\sqrt{\log
p}})/{\rho_n}$.

\noindent \ding{71}
 In the optimal case \eqref{Best Distortion}, this latter becomes:
\begin{equation}\label{eq:MultFact}
C_{mult}= C\,\cdot\,\frac{p\,(1+\log(p/n))\,\sqrt{\log
p}}{n\,\rho_n}\,,
\end{equation}
where $C>0$ is the same universal constant as in \eqref{Best Distortion}. Roughly speaking, up to a
factor of the order of \eqref{eq:MultFact}, the lasso is as good as the oracle that knows the
$S_0$-best term approximation of the target. Moreover, as mentioned in
\eqref{eq:SparsityLevelDistortion}, $S_0$ is an optimal sparsity level. However, this
multiplicative constant takes small values for a restrictive range of the parameter $n$. As a
matter of fact, it is meaningful when $n$ is a constant fraction of $p$.

Similarly, we shows oracle inequalities in error prediction in terms of the distortion of the
kernel of the design.
\begin{theorem}\label{Theo13}
 Let $X\in\R^{n\times p}$ be a full column rank matrix. Assume that $X$ satisfies
$\mathrm{UDP}(S_0,\kappa_0, \Delta )$ and that \eqref{Borne Az} holds. Then for any
\begin{equation*}\tag{\ref{Borne parametre lambda}}
 \lambda_\ell>{\lambda_n^0}/{(1-2\kappa_0)},
\end{equation*}
it holds
\begin{equation}\label{PredictionEll1}
\norm{\ell_2}{X\beta^\ell-X\betastar}\leq\min_{\substack{\S\subseteq\{1,\dotsc,p\},\\ \abs{\S}=s,\
s\leq S_0.}}
\Bigg[4\lambda_\ell\,\Delta\sqrt s+\frac{\norm{\ell_1}{\betastar_{\S^c}}}{\Delta\sqrt s}\Bigg].
\end{equation}
\noindent
--- For every full column
rank matrix $X\in\R^{n\times p}$, for all $0<\kappa_0<1/2$ and $\lambda_\ell$ satisfying
\eqref{Borne
parametre lambda}, we have
\begin{equation}\label{PredictionEll12}
\norm{\ell_2}{X\beta^\ell-X\betastar}\leq\min_{\substack{\S\subseteq\{1,\dotsc,p\},\\
\abs{\S}=s,\\
s\leq ({\kappa_0}/{\delta})^2p.}}
\Bigg[4\lambda_\ell\cdot\frac{2\,\delta}{\rho_n}\cdot\sqrt
s+\frac1{2\delta\sqrt s}\,\cdot\,\rho_n\,{\norm{\ell_1}{\betastar_{\S^c}}} \Bigg],
\end{equation}
where $\rho_n$ denotes the smallest singular value of $X$ and $\delta$ the distortion of its kernel.
\end{theorem}
\noindent \ding{71}
Consider the case where the noise satisfies the hypothesis of Lemma \ref{Bound on the noise} and
take $\lambda_n^0=\lambda_n^0(1)$. Assume that $\kappa_0$ is constant (say $\kappa_0=1/3$)
and take
$\lambda_\ell=3\lambda_n^0$ then \eqref{PredictionEll1} becomes
\[\norm{\ell_2}{X\beta^\ell-X\betastar}\leq
\min_{\substack{\S\subseteq\{1,\dotsc,p\},\\ \abs{\S}=s,\ s\leq S_0.}}
\Bigg[24\,\norm{\ell_2,\infty}{X}\cdot\Delta\sqrt{\log
p}\cdot\sigma_n\,\sqrt s+\frac{\norm{\ell_1}{\betastar_{\S^c}}}{\Delta\sqrt
s}\Bigg],\]
which is not an oracle inequality \textit{stricto sensu} because of $1/(\Delta\sqrt
s)$ in the second term. As a matter of fact, it tends to lower the $s$-best term
approximation term $\norm{\ell_1}{\betastar_{\S^c}}$. Nevertheless, it is \enquote{almost} an
oracle
inequality up to a multiplicative factor of the order of $\Delta\sqrt{\log
p}$. In the same way, \eqref{PredictionEll12} becomes
\[\norm{\ell_2}{X\beta^\ell-X\betastar}\leq
\min_{\substack{\S\subseteq\{1,\dotsc,p\},\\ \abs{\S}=s,\\ s\leq p/9{\delta}^2.}}
\Bigg[48\,\norm{\ell_2,\infty}{X}\cdot\frac{\delta\sqrt{\log
p}}{\rho_n}\cdot\frac1{\rho_n}\,\sigma_n\, \sqrt
s+\frac1{2\delta\sqrt s}\,\cdot\,\rho_n\,{\norm{\ell_1}{\betastar_{\S^c}}} \Bigg],\]
which is an oracle inequality up to a multiplicative factor $C'_{mult}:=({\delta\sqrt{\log
p}})/{\rho_n}$.

\noindent \ding{71}
 In the optimal case \eqref{Best Distortion}, this latter becomes:
\begin{equation}\label{eq:MultFactprime}
C'_{mult}= C\,\cdot\,\frac{({p\log p\,(1+\log(p/n))})^{1/2}}{\rho_n\,\sqrt n}\,,
\end{equation}
where $C>0$ is the same universal constant as in \eqref{Best Distortion}.

\subsection{Results for the Dantzig selector}
Similarly, we derive the same results for the Dantzig selector. The only difference is that the
parameter $\kappa_0$ must be less than $1/4$. Here again the results fold into two parts. In the
first one, we only assume that UDP holds. In the second, we invoke Theorem \ref{Lemma Universality}
to derive results in terms of the distortion of the design.
\begin{theorem}\label{Theo22}
 Let $X\in\R^{n\times p}$ be a full column rank matrix. Assume that $X$ satisfies
$\mathrm{UDP}(S_0,\kappa_0, \Delta )$ with $\kappa_0<1/4$ and that \eqref{Borne Az} holds. Then
for any
\begin{equation}\label{Borne parametre lambda Dantz}
 \lambda_d>{\lambda^0}/{(1-4\kappa_0)},
\end{equation}
it holds
\begin{equation}\label{ConsistEll1Dantz}
\norm{\ell_1}{\beta^d-\betastar}\leq\frac{4}{\big(1-\frac{\lambda^0}{\lambda_d}
\big)-4\kappa_0
} \
\min_{\substack{\S\subseteq\{1,\dotsc,p\},\\ \abs{\S}=s,\ s\leq S_0.}}
\Big(\lambda_d\,\Delta^2\,s+\norm{\ell_1}{\betastar_{\S^c}}\Big).
\end{equation}
\noindent
--- For every full column
rank matrix $X\in\R^{n\times p}$, for all $0<\kappa_0<1/4$ and $\lambda_d$ satisfying
\eqref{Borne parametre lambda Dantz}, we have
\begin{equation}\label{ConsistEll12Dantz}
\norm{\ell_1}{\beta^d-\betastar}\leq\frac{4}{\big(1-\frac{\lambda^0}{\lambda_d}
\big)-4\kappa_0
} \
\min_{\substack{\S\subseteq\{1,\dotsc,p\},\\ \abs{\S}=s,\\ s\leq ({\kappa_0}/{\delta})^2p.}}
\Big(\lambda_d\cdot\frac{4\,\delta^2}{\rho_n^2}\cdot s+\norm{\ell_1}{
\betastar_ { \S^c } } \Big),
\end{equation}
where $\rho_n$ denotes the smallest singular value of $X$ and $\delta$ the distortion of its kernel.
\end{theorem}
\noindent The prediction error is given by the following theorem.

\begin{theorem}\label{Theo23}
 Let $X\in\R^{n\times p}$ be a full column rank matrix. Assume that $X$ satisfies
$\mathrm{UDP}(S_0,\kappa_0, \Delta )$ with $\kappa_0<1/4$ and that \eqref{Borne Az} holds. Then
for any
\begin{equation*}\tag{\ref{Borne parametre lambda Dantz}}
 \lambda_d>{\lambda^0}/{(1-4\kappa_0)},
\end{equation*}
it holds
\begin{equation}\label{PredictionEll1Dantz}
\norm{\ell_2}{X\beta^d-X\betastar}\leq\min_{\substack{\S\subseteq\{1,\dotsc,p\},\\ \abs{\S}=s,\
s\leq S_0.}}
\Bigg[4\lambda_d\,\Delta\sqrt s+\frac{\norm{\ell_1}{\betastar_{\S^c}}}{\Delta\sqrt s}\Bigg].
\end{equation}
\noindent
--- For every full column
rank matrix $X\in\R^{n\times p}$, for all $0<\kappa_0<1/4$ and $\lambda_d$ satisfying
\eqref{Borne parametre lambda}, we have
\begin{equation}\label{PredictionEll12Dantz}
\norm{\ell_2}{X\beta^d-X\betastar}\leq\min_{\substack{\S\subseteq\{1,\dotsc,p\},\\
\abs{\S}=s,\\
s\leq ({\kappa_0}/{\delta})^2p.}}
\Bigg[4\lambda_d\cdot\frac{2\,\delta}{\rho_n}\cdot\sqrt
s+\frac1{2\delta\sqrt s}\,\cdot\,\rho_n\,{\norm{\ell_1}{\betastar_{\S^c}}} \Bigg],
\end{equation}
where $\rho_n$ denotes the smallest singular value of $X$ and $\delta$ the distortion of its kernel.
\end{theorem}

\noindent
Observe that the same comments as in the lasso case (e.g.
\eqref{eq:MultFact}, \eqref{eq:MultFactprime}) hold. Eventually, every deterministic construction of almost-Euclidean sections gives design that satisfies the oracle inequalities above.

\section{An overview of the standard results}\label{Overview}
Oracle inequalities for the lasso and the Dantzig selector have been established under a variety of different conditions on the design. In this section, we show that the UDP condition is comparable to the standard conditions (RIP, REC and Compatibility) and that our results are relevant in the literature on the high-dimensional regression.
\subsection{The standard conditions}
We recall some sufficient conditions here. For all $s\in\{1,\dotsc,p\}$, we denote by $\Sigma_s\subseteq\R^p$ the set of all the $s$-sparse vectors. 
\begin{description}
 \item[\ding{70} Restricted Isoperimetric Property]
A matrix $X\in\R^{n\times p}$ satisfies $RIP(\theta_S)$ if and only if there exists $0<\theta_S<1$ (as small as possible) such that for all $s\in\{1,\dotsc,S\}$, for all $\forall \gamma\in\Sigma_s$, it holds 
\[
(1-\theta_S)\norm{\ell_2}{\gamma}^2\leq \norm{\ell_2}{X\gamma}^2\leq(1+\theta_S)\norm{\ell_2}{\gamma}^2. 
\]
The constant $\theta_S$ is called the $S$-restricted isometry constant.
  \item[\ding{70} Restricted Eigenvalue Assumption \cite{MR2533469}]
A matrix $X$ satisfies $RE(S,c_0)$ if and only if
\begin{equation*}
 \kappa(S,c_0) = \min_{\substack{\S\subseteq\{1,\dotsc,p\}\\\abs \S\leq
S}}\min_{\substack{\gamma\neq 0\\\lVert{\gamma_{\S^c}}\lVert_{\ell_1}\leq
c_0\lVert{\gamma_{\S}}\lVert_{\ell_1}}}\frac{\lVert{X\gamma}
\lVert_{\ell_2}}{\lVert{\gamma_{\S}}\lVert_{\ell_2}}>0\,.
\end{equation*}
The constant $\kappa(S,c_0)$ is called the $(S,c_0)$-restricted $\ell_2$-eigenvalue.
  \item[\ding{70} Compatibility Condition \cite{MR2576316}]
We say that a matrix $X\in\R^{n\times p}$ satisfies the condition $Compatibility(S,c_0)$ if and only if
\begin{equation*}
 \phi(S,c_0) = \min_{\substack{\S\subseteq\{1,\dotsc,p\}\\\abs \S\leq
S}}\min_{\substack{\gamma\neq 0\\\lVert{\gamma_{\S^c}}\lVert_{\ell_1}\leq
c_0\lVert{\gamma_{\S}}\lVert_{\ell_1}}}\frac{\sqrt{\abs \S}\lVert{X\gamma}
\lVert_{\ell_2}}{\lVert{\gamma_{\S}}\lVert_{\ell_1}}>0\,.
\end{equation*}
The constant $\phi(S,c_0)$ is called the $(S,c_0)$-restricted $\ell_1$-eigenvalue.
   \item[\ding{70} $\mathbf H_{S,1}$ Condition \cite{JN10}]
$X\in\R^{n\times p}$ satisfies the $\mathbf H_{S,1}(\kappa)$ condition (with $\kappa<1/2$) if and only if  for all $\gamma\in\R^p $ and for all $\S\subseteq\{1,\dotsc,p\}$ such that $\abs \S\leq S$, it  holds
 \begin{equation}
 \norm{\ell_1}{\gamma_{\mathcal S}}\leq\hat\lambda\,S\,\lVert
 X\gamma\lVert_{\ell_2}+\kappa\lVert \gamma\lVert_{\ell_1},
 \end{equation}
 where $\hat\lambda$ denotes the maximum of the $\ell_2$-norms of the columns in $X$.
\end{description}
\begin{remark}
The first term of the right hand side (i.e. $s\,\lVert  X\gamma\lVert_{\ell_2}$) is greater than the first term of the right hand side of the UDP condition (i.e. $\sqrt s\,\lVert  X\gamma\lVert_{\ell_2}$). Hence the $\mathbf H_{S,1}$ condition is weaker than the $\mathrm{UDP}$ condition. Nevertheless, the authors \cite{JN10} established limits of performance on their conditions: the condition $\mathbf H_{s,\infty}(1/3)$ (that implies  $\mathbf{H}_{s,1}(1/3)$) is feasible only in a severe restricted range of the sparsity parameter $s$. Notice that this is not the case of the $\mathrm{UDP}$ condition, the equality \eqref{eq:SparsityLevelDistortion} shows that it is feasible for a large range of the sparsity parameter $s$. Moreover, a comparison of the two approaches is given in Table \ref{Comparizon}.
\end{remark}
\noindent
Let us emphasize that the above description is not meant to be exhaustive. In particular we do not mention the irrepresentable condition \cite{MR2274449} which ensures exact recovery of the support. 

The next proposition shows that the UDP condition is weaker than the RIP, RE and Compatibility conditions.
\begin{proposition}\label{prop:RIPimpliesUDP}
 Let $X\in\R^{n\times p}$ be a full column rank matrix; then the following is true:
\begin{itemize}
\item[\ding{70}] The $\RIP(\theta_{5S})$ condition with $\theta_{5S}<\sqrt 2-1$ implies
$\mathrm{UDP}(S,\kappa_0 , \Delta)$ for all pairs
$(\kappa_0 , \Delta) $ such that
\begin{equation}\label{eq:RIPimpliesUDPloose}
\bigg[{1+2\Big[\frac{1-\theta_{5S}}{1+\theta_{5S}}\Big]^{\frac12}}\bigg]^{-1}\!\!<\kappa_0<\frac12
\,,\ \mathrm { and}\
\Delta
\geq\bigg[\sqrt{1-\theta_{5S}}+\frac{\kappa_0\!-\!1}{2\kappa_0}\sqrt{1+\theta_{5S}}
\bigg]^ { -1}\,.
\end{equation}
\item[\ding{70}] The $RE(S,c_0)$ condition implies $\mathrm{UDP}(S,c_0,\kappa(S,c_0)^{-1})$.
\item[\ding{70}] The $Compatibility(S,c_0)$ condition implies
$\mathrm{UDP}(S,c_0,\phi(S,c_0)^{-1})$.
\end{itemize}
\end{proposition}
\begin{remark}
The point here is to show that the UDP condition is similar to the standard conditions of the high-dimensional regression. For the sake of simplicity, we do not study that if the converse of Proposition \ref{prop:RIPimpliesUDP} is true. As a matter of fact, the UDP, RE and Compatibility conditions are expressions with the same flavor: they aim at controlling the eigenvalues of $X$ on a cone:
 \[
  \{
  \gamma\in\R^p\quad |\quad \forall\,S\in\{1,\ldots,p\}\,,\ \mathrm{s.t.}\ \abs S\leq s\,,\ \norm{\ell_1}{\gamma_{S^c}}\leq c\norm{\ell_1}{\gamma_S}
  \}\,,
 \]
 where $c>0$ is a tunning parameter.
\end{remark}
\subsection{The results}
Table \ref{Comparizon} shows that our results are similar to standard results in the literature.
\begin{table}[!t]
{
 \begin{tabular}{|c|c||c|c|}
 \hline 
 Reference & Condition & Risk & Prediction\\
 \hline
  \hline
 \cite{MR2543688} & RIP &$\ell_2^2\lesssim \sigma^2{s\log p}+\ell_2^2$ & $L_2^2\lesssim \sigma^2{s\log p} + L_2^2$ \\
 \hline
 \cite{MR2533469} & REC &$\ell_1\lesssim \sigma s\sqrt{\log p}$ (*)  &$L_2\lesssim \sigma\sqrt{s\log p}$ (*)   \\
 \hline 
 \cite{MR2576316} & Comp. & $\ell_1\lesssim \sigma s\sqrt{\log p}$ (*) & $L_2\lesssim \sigma\sqrt{s\log p}$ (*) \\
 \hline
 \cite{JN10}&$\mathbf H_{S,1}$ & $\ell_1\lesssim \sigma s^2\sqrt{\log p} + \ell_1$ & No \\
 \hline
 \hline This article & UDP & $\ell_1\lesssim \sigma s\sqrt{\log p} + \ell_1 $ & $L_2\lesssim \sigma\sqrt{s\log p} + \ell_1/\sqrt s $\\ 
 \hline
\end{tabular}}

\bigskip 
where notations are given by:
\bigskip
{
\begin{tabular}{|c|c|c|c|}
 \hline 
 Location & $\ell_1$ & $\ell_2$ & $L_2$ \\
 \hline
  \hline
 Right hand side & $\norm{\ell_1}{\betastar_{\mathcal S^c}}$ &  $\norm{\ell_2}{\betastar_{\mathcal S^c}}$  &  $\norm{\ell_2}{X\beta^{idea\ell}-X\betastar}$ \\
 \hline
 Left hand side & $\norm{\ell_1}{\hat\beta-\betastar}$ &  $\norm{\ell_2}{\hat\beta-\betastar}$  &  $\norm{\ell_2}{X\hat\beta-X\betastar}$   \\ 
 \hline
\end{tabular}}
\bigskip

 \caption{
 Comparison of results in risk and prediction for the Lasso and the Dantzig selector. Observe that all the inequalities are satisfied with an overwhelming probability. The $\lesssim $ notation means that the inequality holds up to a multiplicative factor that may depends on the parameters of the condition. The (*) notation means that the result is given for $s$-sparse targets. The $\hat\beta$ notation represents the estimator (i.e. the lasso or the Dantzig selector). The parameters $\sigma$ and $p$ represent respectively the standard deviation of the noise and the dimension of the target vector $\betastar$.
  }\label{Comparizon}
\end{table}
\appendix
\setcounter{equation}{0}
\renewcommand{\theequation}{A.\arabic{equation}}
 \section{Appendix}\label{Appendix}
\noindent
The appendix is devoted to the proof of the different results of this paper.
\begin{lemma}\label{Bound on the noise}
 Suppose that $\varepsilon=(\varepsilon_i)_{i=1}^n$ is such that the $\varepsilon_i$'s are i.i.d with respect to a Gaussian distribution with mean zero and variance $\sigma_n^2$. Choose $t\geq 1$ and set
\[\lambda^0_n(t)=(1+t)\cdot\norm{\ell_2,\infty}{X}\cdot \sigma_n\cdot\sqrt{{\log p}},\]
where $\norm{\ell_2,\infty}{X}$ denotes the maximum $\ell_2$-norm of the columns of $X$. Then,
\[\mathbb P\big(\norm{\ell_\infty}{X^\top
\varepsilon}\leq\lambda^0_n(t)\big)\geq1-{\sqrt
2}/\Big[{(1+t)\sqrt{\pi\log p}\  p^{\frac{(1+t)^2}2-1}}\Big]\,.\]
\end{lemma}
\noindent

\begin{pfof}{Lemma \ref{Bound on the noise}}
 Observe that $X^\top \varepsilon\sim\mathcal{N}_p(0,\sigma_n^2\,X^\top X)$. Hence, $\forall j=1,\dotsc,p,\quad X_j^\top
\varepsilon\sim\mathcal{N}\big(0,{\sigma_n^2}\,\norm{\ell_2}{X_j}^2\big)$. Using \v{S}id\'{a}k's inequality \cite{MR0230403}, it yields
\[\mathbb{P}\big(\lVert{X^\top
\varepsilon}\lVert_{\ell_\infty}\leq\lambda_n^0\big)\geq\mathbb{P}\big(\norm{\ell_\infty}{ \widetilde \varepsilon}
\leq\lambda_n^0\big)=\prod_{i=1}^p\mathbb{P}\left(\abs{\widetilde{\varepsilon}_i}
\leq\lambda_n^0\right)\,,\]
where the $\widetilde{\varepsilon}_i$'s are i.i.d. with respect to
$\mathcal{N}\big(0,\sigma_n^2\,\norm{\ell_2,\infty}{X}^2\big)$. Denote by $\Phi$ and $\varphi$
respectively the {cumulative distribution function} and the probability density
function of the standard normal. Set
$\theta=(1+t)\sqrt{\log p}$. It holds
\[
 \prod_{i=1}^p\mathbb{P}\left(\abs{ \widetilde \varepsilon_i}\leq\lambda_n^0\right)
=\mathbb{P}\left(\abs{\varepsilon_1}\leq\lambda_n^0\right)^p =  (2\Phi(\theta)-1)^p >
\big(1-2{\varphi(\theta)}/{\theta}\big)^p\,,
\]
using an integration by parts to get $1-\Phi(\theta)<{\varphi(\theta)}/ \theta$. It yields that
\begin{equation*}
 \mathbb{P}\big(\norm{\ell_\infty}{X^\top
\varepsilon}\leq\lambda_n^0\big)\geq\big(1-2{\varphi(\theta)}/{\theta}
\big)^p
\geq1-2p\frac{\varphi(\theta)}{\theta}=1-\frac{\sqrt
2}{{(1+t)\sqrt{\pi\log p}\  p^{\frac{(1+t)^2}2-1}}}\,.
\end{equation*}
This concludes the proof.
\end{pfof}

\noindent
\begin{pfof}{Theorem \ref{Lemma Universality}}
Consider the following singular value decomposition $ X=U^\top DA$ where
\begin{itemize}
 \item[\ding{71}] $U\in\R^{n \times n}$ is such that $UU^\top =\Id_n$,
 \item[\ding{71}] $D=\mathrm{Diag}(\rho_1,\dotsc,\rho_n)$ is a diagonal matrix where
${\rho_1}\geq\dotsb\geq{\rho_n}> 0$ are the singular values of $X$,
  \item[\ding{71}] and $A\in\R^{n\times p}$ is such that $AA^\top=\Id_n$.
\end{itemize}
We recall that the only assumption on the design is that it has full column rank which yields
that $\rho_n>0$. Let $\delta$ be the distortion of the kernel $\Gamma$ of the
design.
Denote by $\pi_\Gamma$ (resp. $\pi_{\Gamma^\perp}$) the $\ell_2$-projection onto $\Gamma$ (resp.
$\Gamma^\perp$). Let $\gamma\in\R^p$; then $\gamma=\pi_\Gamma(\gamma)+\pi_{\Gamma^\perp}(\gamma)$.
An easy calculation
shows that $\pi_{\Gamma^\perp}(\gamma)=A^\top A \gamma$. Let $ s\in\{1,\dotsc,S\}$ and let $\S\subseteq\{1,\dotsc,p\}$ be such that $\abs\S=s$. It holds,
{\allowdisplaybreaks
 \begin{align*}
  \norm{\ell_1}{\gamma_\S} &\leq\sqrt s\norm{\ell_2}{\gamma} =  \sqrt s \,\norm{\ell_2}{\pi_\Gamma(\gamma)} +
\sqrt s\, \norm{\ell_2}{\pi_{\Gamma^\perp}(\gamma)},\\
& \leq  {\frac{\sqrt s}{\sqrt p}}\,\delta\, \norm{\ell_1}{\pi_\Gamma(\gamma)} +
\sqrt s \,\norm{\ell_2}{A^\top A \gamma},\\
& \leq  {\frac{\sqrt s}{\sqrt p}}\,\delta\,\big(
\norm{\ell_1}{\gamma}+\norm{\ell_1}{(\pi_{\Gamma^\perp}(\gamma))}\big) +
\sqrt s\, \norm{\ell_2}{A \gamma},\\
& \leq  {\frac{\sqrt s}{\sqrt p}}\,\delta\,\norm{\ell_1}{\gamma}+ \delta\,\sqrt s\,
\norm{\ell_2}{A^\top A
\gamma} +\sqrt s\, \norm{\ell_2}{A
\gamma},\\
& \leq  {\frac{\sqrt s}{\sqrt p}}\,\delta\,\norm{\ell_1}{\gamma}+ (1+\delta)\,\sqrt s\,
\norm{\ell_2}{A
\gamma},\\
& \leq  {\frac{\sqrt s}{\sqrt p}}\,\delta\,\norm{\ell_1}{\gamma}+ \frac{1+\delta}{\rho_n}\,\sqrt
s
\,\norm{\ell_2}{X \gamma},\\
& \leq  {\frac{\sqrt s}{\sqrt p}}\,\delta\,\norm{\ell_1}{\gamma}+ \frac{2\delta}{\rho_n}\,\sqrt s
\,\norm{\ell_2}{X \gamma},
 \end{align*}}
 \noindent using the triangular inequality and the distortion of the kernel $\Gamma$. Eventually, set
$\kappa_0={({\sqrt
S}/{\sqrt p})}\,\delta$ and $\Delta={2\delta}/{\rho_n}$. This ends the
proof.
\end{pfof}

\noindent
\begin{pfof}{Theorem \ref{Theo12}}
We recall that $\lambda_n^0$ denotes an upper
bound on the amplification of the noise; see \eqref{Borne Az}. We begin with
a standard result.
\begin{lemma}
 Let $h=\beta^\ell-\betastar\in\R^p$ and $\lambda_\ell\geq\lambda_n^0$. Then, for all subsets
$\S\subseteq\{1,\dotsc,p\}$, it holds,
\begin{equation}\label{StandardInequality}
 \frac{1}{2\lambda_\ell}\Big[\frac12\norm{\ell_2}{Xh}^2+(\lambda_\ell-\lambda_n^0)\norm{\ell_1
}{h}\Big]
\leq \norm{\ell_1}{h_\S}+\norm{\ell_1}{\betastar_{\S^c}}.
\end{equation}
\end{lemma}
\begin{proof}
 By optimality, we have
\[\frac12\norm{\ell_2}{
X\beta^\ell-y }^2+\lambda_\ell \norm{\ell_1}{\beta^\ell}\leq\frac12\norm{\ell_2}{
X\betastar-y }^2+\lambda_\ell \norm{\ell_1}{\betastar}.\]
It yields
 \[\frac12\norm{\ell_2}{Xh }^2-\big\langle X^\top \varepsilon,h\big\rangle+\lambda_\ell
\norm{\ell_1}{\beta^\ell}\leq \lambda_\ell \norm{\ell_1}{\betastar}.\]
Let $\S\subseteq\{1,\dotsc,p\}$; we have
\begin{align*}
 \frac12\norm{\ell_2}{Xh }^2+\lambda_\ell \norm{\ell_1}{\beta^\ell_{\S^c}}& \leq  \lambda_\ell
\big(\norm{\ell_1}{\betastar_\S}-\norm{\ell_1}{\beta^\ell_{\S}}\big)+\lambda_\ell\norm{\ell_1}{
\betastar_{\S^c}}+\big\langle X^\top \varepsilon,h\big\rangle,\\
 &\leq  \lambda_\ell\norm{\ell_1}{h_{\S}} + \lambda_\ell\norm{\ell_1}{
\betastar_{\S^c}} + \lambda_n^0\norm{\ell_1}{h},
\end{align*}
using \eqref{Borne Az}. Adding $\lambda_\ell\norm{\ell_1}{\betastar_{\S^c}}$ on both sides, it
holds
\[\frac12\norm{\ell_2}{Xh }^2+(\lambda_\ell-\lambda_n^0) \norm{\ell_1}{h_{\S^c}}\leq
(\lambda_\ell+\lambda_n^0)\norm{\ell_1}{h_{\S}} +2 \lambda_\ell\norm{\ell_1}{
\betastar_{\S^c}}.
 \]
Adding $(\lambda_\ell-\lambda_n^0) \norm{\ell_1}{h_{\S}} $ on both sides, we conclude the
proof.
\end{proof}
\noindent
Using \eqref{Def UDP} and \eqref{StandardInequality}, it follows that
\begin{equation}\label{eq:DemoUDPOracle}
 \frac{1}{2\lambda_\ell}\Big[\frac12\norm{\ell_2}{Xh}^2+(\lambda_\ell-\lambda_n^0)\norm{\ell_1
}{h}\Big
]
\leq\Delta\sqrt
s\,\norm{\ell_2}{X h}+\kappa_0\norm{\ell_1}{h}+\norm{\ell_1}{\betastar_{\S^c}}.
\end{equation}
It yields,
\begin{align*}
 \Big[\frac12\Big(1-\frac{\lambda_n^0}{\lambda_\ell}\Big)-\kappa_0\Big]\norm{\ell_1}{h} & \leq
\Big(\!-\frac{1}{4\lambda_\ell}\norm{\ell_2}{Xh}^2+\Delta\sqrt
s\,\norm{\ell_2}{X h}\Big)+\norm{\ell_1}{\betastar_{\S^c}},\\
 & \leq
\lambda_\ell\,\Delta^2\,s+\norm{\ell_1}{\betastar_{\S^c}},
\end{align*}
using the fact that the polynomial $x\mapsto -(1/4\lambda_\ell)\,x^2+\Delta \sqrt s\, x$ is not
greater than $\lambda_\ell\,\Delta^2\,s$. This concludes the proof.
\end{pfof}

\noindent
\begin{pfof}{Theorem \ref{Theo22}}
 We begin with
a standard result.
\begin{lemma}
 Let $h=\beta^\ell-\betastar\in\R^p$ and $\lambda_\ell\geq\lambda_n^0$. Then, for all subsets
$\S\subseteq\{1,\dotsc,p\}$, it holds,
\begin{equation}\label{StandardInequalityDantz}
 \frac{1}{4\lambda_d}\Big[\norm{\ell_2}{Xh}^2+(\lambda_d-\lambda_n^0)\norm{\ell_1
}{h}\Big]
\leq \norm{\ell_1}{h_\S}+\norm{\ell_1}{\betastar_{\S^c}}.
\end{equation}
\end{lemma}
\begin{proof}
 Set $h=\beta^\star-\beta^d$. Recall that $\lVert{X^\top \varepsilon}\lVert_{\ell_\infty}\leq\lambda_n^0$,
it yields
\begin{align}
 \lVert{Xh }\lVert_{\ell_2}^2 & \leq
\lVert{X^\top Xh }\lVert_{\ell_\infty}\lVert h \lVert_{\ell_1}  =
\lVert{X^\top \big(y-X\beta^d\big)+X^\top
\big(X\beta^\star-y\big)}\lVert_{\ell_\infty}\norm{\ell_1} h
 \notag\\
& \leq
(\lambda_d+\lambda_n^0)\norm{\ell_1} h \,.\notag
\end{align}
Hence we get
\begin{equation}\label{Xgamma}
 \lVert{Xh
}\lVert_{\ell_2}^2-(\lambda_d+\lambda_n^0)\norm{\ell_1}{h_{S^c}}\leq(\lambda_d+\lambda_n^0)\norm{
\ell_1 }{h _{S}}\,.
\end{equation}
Since $\betastar$ is feasible, it yields $\norm{\ell_1}{\beta^d}\leq\norm{\ell_1}{\beta^\star}$.
Thus,
\[
 \lVert{\beta^d_{S^c}}\lVert_{\ell_1}
\leq\big(\lVert{\beta^\star_S}\lVert_{\ell_1}-\lVert{\beta^d_{S}}\lVert_{\ell_1}\big)+\lVert{
\beta^\star_{ S^c}}\lVert_{\ell_1}\leq\lVert{
h _{S}}\lVert_{\ell_1}+\lVert{\beta^\star_{S^c}}\lVert_{\ell_1}\,.
\]
Since $\norm{\ell_1}{h_{S^c}}\leq\norm{\ell_1}{\beta^d_{S^c}}+\norm{\ell_1}{\beta^\star_{S^c}}$,
it yields
\begin{equation}\label{Dantzig Uncertainty}
 \norm{\ell_1}{h _{S^c}}\leq\norm{{\ell_1}h _{S}}+2\norm{\ell_1}{\beta^\star_{S^c}}\,.
\end{equation}
Combining $\eqref{Xgamma}+2\lambda_d\cdot\eqref{Dantzig Uncertainty}$, we get
\[ \lVert{Xh
}\lVert_{\ell_2}^2+(\lambda_d-\lambda_n^0)\norm{\ell_1}{h_{S^c}}\leq(3\lambda_d+\lambda_n^0)\norm{
\ell_1}{h _{S}}+4\lambda_d\norm{\ell_1}{\beta^\star_{S^c}}\,.\]
Adding $(\lambda_d-\lambda_n^0) \norm{\ell_1}{h_{\S}} $ on both sides, we conclude the
proof.
\end{proof}
Using \eqref{Def UDP} and \eqref{StandardInequalityDantz}, it follows that
\begin{equation}\label{eq:DemoUDPOracleDantz}
 \frac{1}{4\lambda_\ell}\Big[\norm{\ell_2}{Xh}^2+(\lambda_\ell-\lambda_n^0)\norm{\ell_1
}{h}\Big
]
\leq\Delta\sqrt
s\,\norm{\ell_2}{X h}+\kappa_0\norm{\ell_1}{h}+\norm{\ell_1}{\betastar_{\S^c}}.
\end{equation}
It yields,
\begin{align*}
 \Big[\frac14\Big(1-\frac{\lambda_n^0}{\lambda_\ell}\Big)-\kappa_0\Big]\norm{\ell_1}{h} & \leq
\Big(\!-\frac{1}{4\lambda_\ell}\norm{\ell_2}{Xh}^2+\Delta\sqrt
s\,\norm{\ell_2}{X h}\Big)+\norm{\ell_1}{\betastar_{\S^c}},\\
 & \leq
\lambda_\ell\,\Delta^2\,s+\norm{\ell_1}{\betastar_{\S^c}},
\end{align*}
using the fact that the polynomial $x\mapsto -(1/4\lambda_\ell)\,x^2+\Delta \sqrt s\, x$ is not
greater than $\lambda_\ell\,\Delta^2\,s$. This concludes the proof.
\end{pfof}

\noindent
\begin{pfof}{Theorem \ref{Theo13} and Theorem \ref{Theo23}}
Using \eqref{eq:DemoUDPOracle}, we know that
\[\frac{1}{2\lambda_\ell}\Big[\frac12\norm{\ell_2}{Xh}^2+(\lambda_\ell-\lambda_n^0)\norm{
\ell_1}{h}\Big
]
\leq\Delta\sqrt
s\,\norm{\ell_2}{X h}+\kappa_0\norm{\ell_1}{h}+\norm{\ell_1}{\betastar_{\S^c}}.\]
It follows that
\[\norm{\ell_2}{Xh}^2-4\lambda_\ell\,\Delta\sqrt s\,\norm{\ell_2}{X h}
\leq4\lambda_\ell\,\norm{\ell_1}{\betastar_{\S^c}}\,.\]
This latter is of the form $x^2-b x\leq c$ which implies that $x\leq
b+c/b$. Hence,
\[\norm{\ell_2}{Xh}\leq4\lambda_\ell\,\Delta\sqrt s
+\frac{\norm{\ell_1}{\betastar_{\S^c}}}{\Delta\sqrt s}\,.\]
The same analysis holds for Theorem \ref{Theo23}.
\end{pfof}

\begin{pfof}{Proposition \ref{prop:RIPimpliesUDP}}
One can check that $RE(S,c_0)$ implies $\mathrm{UDP}(S,c_0,\kappa(S,c_0)^{-1})$, and that $Compatibility(S,c_0)$
implies $\mathrm{UDP}(S,c_0,\phi(S,c_0)^{-1})$.

Assume that $X$ satisfies $\RIP(\theta_{5S})$. Let $\gamma\in\R^p$, $
s\in\{1,\dotsc,S_0\}$, and
$T_0\subseteq\{1,\dotsc,p\}$ such that $\abs{T_0}=s$. Choose a pair $(\kappa_0 , \Delta) $ as in
\eqref{eq:RIPimpliesUDPloose}.

\noindent\ding{71} If
$\norm{\ell_1}{\gamma_{T_0}}\leq\kappa_0\norm{\ell_1}{\gamma}$ then
$\norm{\ell_1}{\gamma_{T_0}}\leq\Delta\sqrt
s\norm{\ell_2}{X\gamma}+\kappa_0\norm{\ell_1}{\gamma}$.

\noindent\ding{71} Suppose that
$\norm{\ell_1}{\gamma_{T_0}}>\kappa_0\norm{\ell_1}{\gamma}$ then
\begin{equation}\label{eq:DemoRIpimpliesUDP1}
 \norm{\ell_1}{\gamma_{T_0^c}}<\frac{1-\kappa_0}{\kappa_0}\norm{\ell_1}{\gamma_{T_0}}\,.
\end{equation}
Denote by $T_1$ the set of the indices of the $4s$ largest coefficients (in absolute value) in
$T_0^c$, denote by $T_2$ the set of the indices of the $4s$ largest coefficients in
$(T_0\cup T_1)^c$, etc... Hence we decompose $T_0^c$ into disjoint sets
$T_0^c=T_1\cup T_2\cup\dotsc\cup T_l$. Using \eqref{eq:DemoRIpimpliesUDP1}, it yields
\begin{equation}\label{eq:ContrainteconeUDP}
\sum_{i\geq 2}\norm{\ell_2}{\gamma_{T_i}}\leq
(4s)^{-1/2}\sum_{i\geq
1}\norm{\ell_1}{\gamma_{T_i}}=(4s)^{-1/2}\norm{\ell_1}{\gamma_{T_0^c}}\leq
\frac{1-\kappa_0}{2\kappa_0\sqrt s}\norm{\ell_1}{\gamma_{ T_0}}
\end{equation}
Using $\RIP(\theta_{5S})$ and \eqref{eq:ContrainteconeUDP}, it follows that
\begin{align*}
 \norm{\ell_2}{X\gamma}& \geq  \norm{\ell_2}{X(\gamma_{(T_0\cup T_1)})}-\sum_{i\geq
2}\norm{\ell_2}{X(\gamma_{T_i})}\,,\\
& \geq  \sqrt{1-\theta_{5S}}\,\norm{\ell_2}{\gamma_{(T_0\cup
T_1)}}-\sqrt{1+\theta_{5S}}\sum_{i\geq
2}\norm{\ell_2}{\gamma_{T_i}}\,,\\
& \geq  \sqrt{1-\theta_{5S}}\,{\norm{\ell_2}{\gamma_{T_0}}}-\sqrt{1+\theta_{5S}}\
\frac{1-\kappa_0}{2\kappa_0}\ \frac{\norm{\ell_1}{\gamma_{T_0}}}{\sqrt s}\,,\\
 & \geq  \Big[\sqrt{1-\theta_{5S}}+\frac{\kappa_0-1}{2\kappa_0}\sqrt{1+\theta_{5S}}
\Big]\,\frac{\norm{\ell_1}{\gamma_{T_0}}}{\sqrt s}\,,\\
 &
=\frac{\sqrt{1+\theta_{5S}}}{2\kappa_0}\bigg[{1+2\Big(\frac{1-\theta_{5S}}{1+\theta_{5S}}
\Big)^ {\frac12
}} \bigg]
\bigg[\kappa_0-\bigg[{1+2\Big(\frac{1-\theta_{5S}}{1+\theta_{5S}}\Big)^{\frac12}}\bigg]^{-1}
\bigg]\,\frac{\norm{\ell_1}{\gamma_{T_0}}}{\sqrt s}\,.
\end{align*}
The lower bound on $\kappa_0$ shows that the right hand side is positive. Observe that we took
$\Delta$ such that this latter is exactly
${\norm{\ell_1}{\gamma_{T_0}}}/( \Delta{\sqrt s})$. Eventually, we get
\[\norm{\ell_1}{\gamma_{T_0}}\leq\Delta\sqrt s\norm{\ell_2}{X\gamma}\leq \Delta\sqrt
s\norm{\ell_2}{X\gamma}+\kappa_0\norm{\ell_1}{\gamma}\,.\]
This ends the proofs.
\end{pfof}


\vspace*{0.5cm}

\noindent\textbf{Acknowledgments} --- The author would like to thank Jean-Marc Aza\"is and Franck Barthe for their support. The authors would like to thank the anonymous reviewers for their valuable comments and suggestions.

 \bibliographystyle{amsalpha}
 \nocite{*}
 \bibliography{Lasso}
\end{document}